\documentclass[10.5pt, reqno]{amsart}
\usepackage[a4paper, margin=30mm]{geometry}
\usepackage{amsmath}
\usepackage{amssymb}
\usepackage{amsthm}
\usepackage{mathrsfs}
\usepackage{url}
\newtheorem{dfn}{Definition}[section]
\newtheorem{thm}[dfn]{Theorem}
\newtheorem{prop}[dfn]{Proposition}
\newtheorem{lem}[dfn]{Lemma}
\newtheorem{rem}[dfn]{Remark}
\newtheorem{cor}[dfn]{Corollary}
\newtheorem{exm}[dfn]{Example}
\usepackage{color}
\usepackage[bookmarks=false,draft=false,breaklinks,colorlinks]{hyperref}

\begin{document}

\title{The Riemannian Median Of Positive-Definite Matrices}
\author{Yutaro Nakagawa}
\address{
Graduate School of Mathematics, Nagoya University, 
Furocho, Chikusaku, Nagoya, 464-8602, Japan
}
\email{yutaro-nakagawa@ymail.ne.jp}

\maketitle


\begin{abstract}
Using Landers and Rogge's work \cite{Lan81} partially, 
we define the Riemannian median $M(\mathbb{A})$ of a tuple of positive-definite matrices $\mathbb{A}=(A_{1}, \cdots, A_{n})$ as a positive-definite matrix, 
not as a set unlike Yang's work \cite{Yan10}. 
Then, in the Riemannian manifold of positive-definite matrices with the trace metric, we show 
\[
\delta(M, \Lambda) \leq \frac{1}{n}\sum_{k=1}^{n}\delta(A_{k}, \Lambda) \leq \sqrt{\frac{1}{n} \sum_{k=1}^{n} \delta(A_{k}, \Lambda)^{2}},  
\]
where $M=M(\mathbb{A})$, $\Lambda$ is the Karcher mean of $\mathbb{A}$,  
and $\delta$ is the Riemannian distance induced by the trace metric.  
This inequality is an analogue of $|\mu-m| \leq \sigma$, where $\mu$, $m$ and $\sigma$ are the mean, the median and the standard deviation 
of real-valued data points. 
Moreover, we investigate the commutative case, 
how outliers have an effect on the Riemannian median, 
the congruence invariance, the joint homogeneity, the self-duality and the monotonicity. 
\end{abstract}


\section{Introduction}

Let $\mathbb{P}_{d}$ denote all the $d \times d$ positive-definite matrices and 
let $\mathbb{A}=(A_{1}, \cdots, A_{n}) \in (\mathbb{P}_{d})^{n}$. 
It is well known that $\mathbb{P}_{d}$ admits the structure of a Riemannian manifold and it allows us to define the length of curves on $\mathbb{P}_{d}$, 
which induces the complete Riemannian distance $\delta$. 
By using this geometric structure on $\mathbb{P}_{d}$, 
the {\it Karcher mean} of $\mathbb{A}=(A_{1}, \cdots, A_{n})$ is defined by a unique minimizer of the function 
\[
F_{2}(X ; \mathbb{A})=\frac{1}{n}\sum_{k=1}^{n}\delta(X, A_{k})^{2}. 
\]
In fact, $\mathbb{P}_{d}$ has the structure of a CAT$(0)$ space \cite[Chapter II]{Bri99}, 
and therefore, the Karcher mean is uniquely determined. 
The Karcher mean was proposed by Moakher in 2005 \cite{Moa05} and by Bhatia and Holbrook in 2006 \cite{Bha062} as a geometric mean of matrices. 
See e.g., \cite{Moa05}, \cite{Bha062} and \cite{Lim122} for more details.  

It seems natural to consider the function 
\[
F_{1}(X ; \mathbb{A})=\frac{1}{n}\sum_{k=1}^{n}\delta(X, A_{k}). 
\]
and regard its minimizer as the ``median'' of $A_{1}, \cdots, A_{n}$. 
In \cite{Bac014}, Ba\v{c}\'{a}k investigated the median in a Hadamard space with emphasis on an algorithm for computing it. 
In \cite{Yan10}, Yang considered the function $f(x)=\int_{M} d_{M}(x, p) \mu(dp)$ 
for a probability measure $\mu$ with appropriate conditions on a Riemannian manifold $M$ equipped with a complete Riemannian distance $d_{M}$, 
and defined the Riemannian median. 
Here is an important notice that the uniqueness of minimizer of the function $F_{1}$, defining the median as expected, unfortunately does not hold in general, 
and hence, the median is usually defined as a ``set''. 
At least for matrices, 
it is more convenient to define the median as a matrix rather than a set; 
however, to the best of the author's knowledge, there has been no such attempt so far. 

The main purposes of these notes are
\begin{itemize}
  \item to investigate in details the cases where the uniqueness of the Riemannian median holds and where it does not (Section 3), 
  \item to define the Riemannian median as a single positive-definite matrix (Section 3) and 
  \item to examine several properties of the Riemannian median (Section 4). 
\end{itemize}
In Section 4, we will investigate several aspects; 
the commutative case, how outliers have an effect on the Riemannian median, 
the congruence invariance, the joint homogeneity, the self-duality and the monotonicity in a special case,  
and will give a counter example showing that the monotonicity of the Riemannian median does not hold in general. 


\section{Preliminaries}

\subsection{CAT$(0)$ space} 

In this subsection, 
we review some basic facts on CAT$(0)$ spaces. 
See e.g., \cite{Bri99} for more details.  

Let $(X, d)$ be a metric space. 
A {\it geodesic} from $x \in X$ to $y \in X$ is a map $\gamma : (\mathbb{R} \supset) [0, l] \to X$ 
such that $\gamma(0)=x$, $\gamma(l)=y$ and $d(\gamma(t), \gamma(t^{\prime}))=|t-t^{\prime}|$ for all $t, t^{\prime} \in [0, l]$. 
A {\it geodesic segment} with endpoints $x$ and $y$ is the image of $\gamma$ and denoted by $[x, y]$.
A subset $C \subset X$ is said to be {\it convex} if for every pair of $x, y \in C$ can be joined by a geodesic in X and the image of every such geodesic is contained in $C$. 
$(X, d)$ is said to be a (uniquely) geodesic metric space if any two points in $X$ are joined by a (unique) geodesic. 

Here is the definition of a CAT$(0)$ space. 

\begin{dfn}[{\cite[Chapter II, Exercises 1.9]{Bri99}}]
Let $(X, d)$ be a geodesic metric space. 
We call $X$ a CAT$(0)$ space, if
\[
d(p, q)^{2}+d(p, r)^{2} \geq 2d(m, p)^{2}+\frac{1}{2}d(q, r)^{2}
\]
for all $p, q, r \in X$ and all $m \in X$ with $d(q, m)=d(r, m)=d(q, r)/2$. 
\end{dfn}

Let $(X, d)$ be a CAT$(0)$ space and $p, q, r$ be three points in $X$. 
It is shown that there are points $\overline{p}, \overline{q}, \overline{r} \in \mathbb{R}^{2}$ such that 
$d(p, q)=d_{\mathbb{R}^{2}}(\overline{p}, \overline{q})$, $d(q, r)=d_{\mathbb{R}^{2}}(\overline{q}, \overline{r})$ and $d(r, p)=d_{\mathbb{R}^{2}}(\overline{r}, \overline{p})$, where $d_{\mathbb{R}^{2}}$ is the Euclidean distance on $\mathbb{R}^{2}$ ({\cite[Chapter I, Definition 1.10, Lemma 2.14]{Bri99}}). 
The triangle with vertices $\overline{p}, \overline{q}, \overline{r}$, denoted by $\overline{\triangle}(p, q, r)$, 
is called a {\it comparison triangle} for the triple $(p, q, r)$. 
It is uniquely determined up to isometries of $\mathbb{R}^{2}$ (\cite[Chapter I, Lemma 2.14]{Bri99}). 
The interior angle of $\overline{\triangle}(p, q, r)$ at $\overline{p}$ 
is called a {\it comparison angle} between $q$ and $r$ at $p$ and is denoted by $\overline{\angle}_{p}(q, r)$. 
The comparison triangles enables us to define the Alexandrov angle ({\cite[Chapter I, Definition 1.12]{Bri99}}).
Let $c : [0, a] \to X$ and $c^{\prime} : [0, a^{\prime}] \to X$ be two geodesics with $c(0)=c^{\prime}(0)$. 
The {\it Alexandrov angle} between $c$ and $c^{\prime}$ is the number $\angle(c, c^{\prime}) \in [0, \pi]$ defined by 
\[
\angle(c, c^{\prime})=\limsup_{t, t^{\prime} \to 0}\overline{\angle}_{c(0)}(c(t), c^{\prime}(t^{\prime})). 
\]
We introduce the notion of an orthogonal projection. 

\begin{prop}[{\cite[Chapter II, Proposition 2.4]{Bri99}}]\label{orthoproj}
Let $(X, d)$ be a CAT$(0)$ space, and let $C$ be a convex subset which is complete in the induced metric. 
Then, 
\begin{enumerate}
  \item for every $x \in X$, there exists a unique point $\mathcal{P}(x) \in X$ such that $d(x, \mathcal{P}(x))=d(x, C):=\inf_{y \in C}d(x, y)$; 
  \item given $x \notin C$ and $y \in C$, if $y \neq \mathcal{P}(x)$ then $\angle_{\mathcal{P}(x)}(x, y) \geq \pi/2$.
\end{enumerate}
\end{prop}
We call the map $\mathcal{P} : X \to C, \ x \mapsto \mathcal{P}(x)$ the {\it orthogonal projection} onto $C$.

Since a CAT$(0)$ space is a uniquely geodesic metric space (\cite[Chapter II, Proposition 1.4]{Bri99}), 
for any three points $p, q, r$ in CAT$(0)$, 
we can consider a {\it geodesic triangle} with vertices $p, q, r$ and edges $[p, q]$, $[q, r]$ and $[r, p]$. 
This geodesic triangle is denoted by $\triangle([p, q], [q, r], [r, p])$, or simply $\triangle(p, q, r)$. 
The following proposition plays an important role in our discussion below. 

\begin{prop}[{\cite[Chapter II, Proposition 1.7]{Bri99}}]\label{Alxcmp}
Let $X$ be a CAT$(0)$ space. 
Then, the Alexandrov angle between the sides of any geodesic triangle in $X$ with distinct vertices is no greater than the 
angle between the corresponding sides of its comparison triangle in $\mathbb{R}^{2}$. 
\end{prop}


\subsection{Jensen's inequality for a probability measure on a global NPC space} 

In this subsection, 
we give a brief exposition of the Jensen's inequality for a probability measure on a global NPC space. 

Let $(N, d)$ be a geodesic space. 
A function $f : N \to \mathbb{R}$ is called {\it convex} if the function $f \circ \gamma : [0, l] \to \mathbb{R}$ is convex for each geodesic $\gamma : [0, l] \to N$. 

\begin{dfn}[{\cite[Definition 2.1]{Str03}}]
A metric space $(N, d)$ is called a global NPC space if it is complete 
and if for each pair of points $x_{0}, x_{1} \in N$ there exists a point $y \in N$ 
with the property that for all points $z \in N$:
\[
d(z, y)^{2} \leq \frac{d(z, x_{0})^{2}+d(z, x_{1})^{2}}{2}-\frac{1}{4}d(x_{0}, x_{1})^{2}. 
\] 
\end{dfn}

It is well known that the space $(\mathbb{P}_{d}, \delta)$ in the last subsection is a global NPC space. 

Let $(N, d)$ be a global NPC space 
and let $\mathcal{P}(N)$ denote the set of probability measures on $(N, \mathcal{B}(N))$ with separable support, 
where $\mathcal{B}(N)$ is the Borel $\sigma$-algebra over $N$. 
For a fixed $\theta \in [0, \infty)$, 
$\mathcal{P}^{\theta}(N)$ denotes the set of $p \in \mathcal{P}(N)$ with $\int_{N} d(x, y)^{\theta} p(dy)<+\infty$ for all $x \in N$. 
Obviously, $\mathcal{P}^{\theta}(N) \subset \mathcal{P}^{1}(N)$ for all $\theta \in [0, \infty)$. 

We recall the notion of {\it barycenter}, which may be regarded as a generalization of the Karcher mean. 
It is shown that there exists a unique point which minimizes the function 
$N \to [0, \infty), \ z \mapsto \int_{N} d(z, x)^{2} p(dx)$ for each $p \in \mathcal{P}^{2}(N)$ (\cite[Proposition 4.3]{Str03}). 
The point is called the {\it barycenter} of $p$ and denoted by $b(p)$.

\begin{prop}[{\cite[Theorem 6.2]{Str03}}]\label{Jensen}
Let $(N, d)$ be a global NPC space. 
For any lower semicontinuous convex function $f : N \to \mathbb{R}$ and $p \in \mathcal{P}^{1}(N)$, 
\[
f(b(p)) \leq \int_{N} f(x) p(dx)
\] 
provided the right-hand side is well-defined. 
\end{prop}


\section{Definition of the Riemannian median}

Yang \cite[Definition 1]{Yan10} defined a Riemannian median 
of a probability measure satisfying certain conditions on a complete Riemannian manifold. 
However, it was defined as a set there. 
In \cite{Lan81}, 
Landers and Rogge introduced a notion of the ``natural median'' of a random variable using a limit of $L^{s}$-mean with $s>1$.  
In this section, based in part on their ideas, we define a Riemannian median of positive-definite matrices as a single positive-definite matrix.   

Before giving its definition, we show some lemmas necessary to define our Riemannian median. 
See e.g., \cite{Bha062} or \cite[Chapter 6]{Bha07} for more details of the Riemannian geometry on $\mathbb{P}_{d}$. 
We define maps $F_{p}(\cdot ; \mathbb{A}) : \mathbb{P}_{d} \to [0, \infty)$ by 
\[
F_{p}(X ; \mathbb{A})=\frac{1}{n}\sum_{k=1}^{n}\delta(X, A_{k})^{p}
\]
for $\mathbb{A}=(A_{1}, \cdots, A_{n}) \in (\mathbb{P}_{d})^{n}$ and $p \geq 1$. 
It is easy to see that $F_{p}(\cdot ; \mathbb{A})$ with $p>1$ has a unique minimizer $M_{p}(\mathbb{A})$ 
by the convexity of $\delta$ and the strictly convexity of $t^{p}$ on $[0, \infty)$. 
We begin by investigating minimizers of $F_{1}(\cdot ; \mathbb{A})$. 

\begin{lem}[{\cite[Lemma 2.3]{Bac014}}]\label{pfofexmnmzr}
A coercive convex lsc real-valued function on a Hadamard space has a minimizer.
\end{lem}

Recall that the {\it $t$-weighted geometric mean} of $A, B \in \mathbb{P}_{d}$ is defined by 
\[
A \sharp_{t} B:=A^{1/2}(A^{-1/2}BA^{-1/2})^{t}A^{1/2}, \ t \in [0, 1].
\]
See e.g., \cite{And042} for more details. 
By Lemma $\ref{pfofexmnmzr}$, 
it is shown that $F_{1}(\cdot ; \mathbb{A})$ has a minimizer.
The next lemma is immediate. 

\begin{lem}\label{geolem}
Let $A, B \in \mathbb{P}_{d}$ with $A \neq B$ and $D=\delta(A, B)$. 
Then, a curve 
\[
\gamma : [0, D] \to \mathbb{P}_{d}, \ \gamma(t)=A \#_{t/D} B
\]
by means of weighted geometric mean is a geodesic in the sense of subsection 2.1, 
\end{lem}

We next discuss the uniqueness of minimizer of $F_{1}(\cdot ; \mathbb{A})$. 
We begin by investigating the case where $A_{1}, \cdots, A_{n}$ do not lie on a common geodesic. 
This case was discussed and shown by Yang in \cite{Yan10}. 

\begin{lem}[{\cite[Theorem 3.1]{Yan10}}]\label{YangLem}
If $A_{1}, \cdots, A_{n}$ do not lie on a common geodesic, 
then a minimizer of $F_{1}(\cdot ; \mathbb{A})$ is unique. 
\end{lem}

We now turn to the case where $A_{1}, \cdots, A_{n}$ lie on a common geodesic, 
in other words, 
the case where there exist $i, j \in \{1, \cdots, n\}$ with the property that 
for all $k=1, \cdots, n$, there exists $t_{k} \in [0, 1]$ such that $A_{k}=A_{i} \#_{t_{k}} A_{j}$. 
Without loss of generality, 
we may assume that $i=1$, $j=n$, $t_{k} \leq t_{k+1}$ for all $k=1, \cdots, n-1$, $t_{1}=0$ and $t_{n}=1$. 

\begin{lem}\label{minongamma}
If $A_{1}, \cdots, A_{n}$ lie on a common geodesic $\gamma : [0, l] \to \mathbb{P}_{d}$, 
then $F_{1}(\cdot ; \mathbb{A})$ has its minimizer(s) on $\gamma([0, l])$.   
\end{lem}

\begin{proof}
Suppose that a minimizer $M=M_{1}(\mathbb{A})$ of $F_{1}(\cdot ; \mathbb{A})$ does not belong to $\gamma([0, l])$ and 
consider a geodesic triangle $\triangle(M, \mathcal{P}(M), A_{k})$ for each $k=1, \cdots, n$,  
where $\mathcal{P}$ is the orthogonal projection onto $C:=\gamma([0, l])$. 
We show that $\delta(M, A_{k}) \geq \delta(\mathcal{P}(M), A_{k})$ for each $k=1, \cdots, n$.  
It is enough to consider the case where $\mathcal{P}(M) \neq A_{k}$.   

By the properties of $\mathcal{P}$ given in \cite[Proposition 2.4]{Bri99}, 
$\angle_{\mathcal{P}(M)}(M, A_{k}) \geq \pi/2$. 
Hence, by Proposition $\ref{Alxcmp}$, $\pi/2 \leq \angle_{\mathcal{P}(M)}(M, A_{k}) \leq \overline{\angle}_{\mathcal{P}(M)}(M, A_{k})$.  
Applying the law of cosines to the comparison triangle $\overline{\triangle}(M, \mathcal{P}(M), A_{k})$, 
we have 
\begin{align*}
&d_{\mathbb{R}^{2}}(\overline{M}, \overline{A_{k}})^{2} \\
&=d_{\mathbb{R}^{2}}(\overline{\mathcal{P}(M)}, \overline{A_{k}})^{2}+d_{\mathbb{R}^{2}}(\overline{\mathcal{P}{(M)}}, \overline{M})^{2}
-\{2d_{\mathbb{R}^{2}}(\overline{\mathcal{P}(M)}, \overline{A_{k}})d_{\mathbb{R}^{2}}(\overline{\mathcal{P}{(M)}}, \overline{M}) \\ 
&\qquad\qquad\qquad\qquad\qquad\qquad\qquad\qquad\quad \ \times \cos{\overline{\angle}_{\mathcal{P}(M)}(M, A_{k})}\} \\
&\geq d_{\mathbb{R}^{2}}(\overline{\mathcal{P}(M)}, \overline{A_{k}})^{2}+d_{\mathbb{R}^{2}}(\overline{\mathcal{P}{(M)}}, \overline{M})^{2} 
\quad ({\rm since} \  \cos\overline{\triangle}(M, \mathcal{P}(M), A_{k}) \leq 0) \\
&>d_{\mathbb{R}^{2}}(\overline{\mathcal{P}(M)}, \overline{A_{k}})^{2}
\end{align*}
for any $k=1, \cdots, n$,  
where $d_{\mathbb{R}^{2}}$ is the Euclidean distance. 
By the definition of a comparison triangle and since $k$ is arbitrary, 
we have $\delta(M, A_{k})>\delta(\mathcal{P}(M), A_{k})$ for all $k=1, \cdots, n$.  
Hence, we have obtained $F(\mathcal{P}(M))<F(M)$, 
a contradiction to the fact that $M$ is a minimizer of $F_{1}(\cdot ; \mathbb{A})$. 
\end{proof}

Just as we distinguish cases based on the parity of the number of data points when considering the median of a finite set of real numbers, 
we must treat the even and odd cases of $n$ separately in the following two lemmas. 
Let $[A, B]$ denote the geodesic joining $A \in \mathbb{P}_{d}$ to $B \in \mathbb{P}_{d}$ with respect to the trace metric. 

\begin{lem}\label{odd}
If $A_{1}, \cdots, A_{n}$ lie on a common geodesic and $n=2m-1$ for some positive integer $m$, 
then then a minimizer of $F_{1}(\cdot ; \mathbb{A})$ is unique. 
\end{lem}

\begin{proof}
As mentioned above, 
without loss of generality, 
we may assume that there exists $t_{k} \in [0, 1]$ such that 
$A_{k}=A_{1} \#_{t_{k}} A_{n}$ for all $k=1, \cdots, n$ 
and $t_{k} \leq t_{k+1}$ for all $k=1, \cdots, n-1$.  
Note that $t_{1}=0$ and $t_{n}=1$. 
Let $D=\delta(A_{1}, A_{n})$. 
Since the case of $D=0$ is trivial ($A_{1}=\cdots=A_{n}=M$), 
it suffices to prove the assertion for the case of $D>0$. 
Let $f$ be the function defined by 
\[
f : [0, D] \to [0, \infty), \ f(t)=\frac{1}{n}\sum_{k=1}^{n} |t-t_{k}|. 
\]
By Lemma $\ref{geolem}$, 
$t \in [0, D]$ minimizes $f$ if and only if $A_{1} \#_{t/D} A_{n}$ minimizes $F_{1}(\cdot ; \mathbb{A})$. 
We rearrange the terms of $f$ as follows: 
\[
f(t)=\frac{1}{2m-1}\sum_{k=1}^{m-1} (|t-t_{k}|+|t-t_{2m-k}|) + |t-t_{m}|. 
\]
Here, we may assume that $m \geq 2$ since $t_{1}$ is a unique minimizer of $f$ when $m=1$. 
We have to minimize all $|t-t_{k}|+|t-t_{2m-k}|$ and $|t-t_{m}|$ to minimize $f$. 
Observe that every point in the closed interval $[t_{k}, t_{2m-k}]$ minimizes $|t-t_{k}|+|t-t_{2m-k}|$ for all $k=1, \cdots, m-1$ and 
$[t_{m-1}, t_{m+1}] \subset \cdots \subset [t_{1}, t_{n}]$ holds.  
Hence, every point in the closed interval $[t_{m-1}, t_{m+1}]$ minimizes $\sum_{k=1}^{m-1} (|t-t_{k}|+|t-t_{2m-k}|)$. 
Since $t_{m}$ is a unique minimizer of $|t-t_{m}|$ and $t_{m} \in [t_{m-1}, t_{m+1}]$, 
$t_{m}$ is a unique minimizer of $f$. 
Therefore, by Lemma $\ref{minongamma}$, 
$A_{1} \#_{t_{m}/D} A_{n} \in [A_{1}, A_{n}]$ is a unique minimizer of $F_{1}(\cdot ; \mathbb{A})$. 
\end{proof}

\begin{lem}\label{even}
If $A_{1}, \cdots, A_{n}$ lie on a common geodesic and $n=2m$ for some positive integer $m$, 
then every point on the geodesic $[A_{m}, A_{m+1}]$ minimizes $F_{1}( \cdot ; \mathbb{A})$. 
\end{lem}

\begin{proof}
Without loss of generality, 
we may assume that $i=1$, $j=n$, $t_{k} \leq t_{k+1}$ for all $k=1, \cdots, n-1$, $t_{1}=0$ and $t_{n}=1$. 
In the present case, 
the rearrangement of the terms of $f$ we carried out in Lemma $\ref{odd}$ now becomes
\[
f(t)=\frac{1}{2m}\sum_{k=1}^{m} (|t-t_{k}|+|t-t_{2m+1-k}|). 
\]
Therefore, every point in the closed interval $[t_{m}, t_{m+1}]$ minimizes $f$, 
which means that every positive-definite matrices on $[A_{m}, A_{m+1}]$ minimizes $F_{1}( \cdot ; \mathbb{A})$. 
\end{proof}

By the lemmas so far, 
$F_{1}( \cdot ; \mathbb{A})$ indeed has a minimizer, which belongs to $[A_{m}, A_{m+1}]$ 
if $A_{1}, \cdots, A_{n}$ lie on a common geodesic and $n=2m$ for some positive integer $m$. 


\begin{lem}\label{convlocuni}
$\{F_{p}( \cdot ; \mathbb{A})\}_{p \in (1, 2]}$ converges uniformly to $F_{1}( \cdot ; \mathbb{A})$ on every compact subset as $p \searrow 1$. 
\end{lem}

\begin{proof}
Let $K \subset \mathbb{P}_{d}$ be a compact subset. 
Since $\delta(\cdot, A_{k})$ is continuous on $K$ for all $k=1, \cdots, n$,  
it has a maximum $R_{k}$ on $K$.  
Putting $R=\max_{i \in \{1, \cdots, n\}}R_{i}$, we obtain $\delta(\cdot, A_{k}) \leq R$ for any $k=1, \cdots n$. 

We define a function $\Phi : [0, R] \times [1, 2] \to \mathbb{R}$ by $\Phi(y, p):=y^{p}$. 
Since $[0, R] \times [1, 2]$ is compact and $\Phi$ is continuous on it, 
$\Phi$ is uniformly continuous. 
Hence, it follows that $\sup_{y \in [0, R]} |\Phi(y, p)-\Phi(y, 1)| \to 0$ as $p \searrow 1$. 
\begin{align*}
|F_{p}(X ; \mathbb{A})-F_{1}(X ; \mathbb{A})| 
\leq \frac{1}{n} \sum_{k=1}^{n} |\delta(X, A_{k})^{p}-\delta(X, A_{k})| 
\leq \sup_{y \in [0, R]} |\Phi(y, p)-\Phi(y, 1)|
\end{align*}
holds for all $X \in K$. 
Note that the right-hand side $\sup_{y \in [0, R]} |\Phi(y, p)-\Phi(y, 1)|$ does not depend on $X$. 
Therefore, $\{F_{p}( \cdot ; \mathbb{A})\}_{p \in (1, 2]}$ converges uniformly to $F_{1}( \cdot ; \mathbb{A})$ on $K$. 
\end{proof}

\begin{lem}\label{mplimexists}
$\lim_{p \searrow 1}M_{p}(\mathbb{A})$ exists and minimizes $F_{1}( \cdot ; \mathbb{A})$. 
\end{lem}

\begin{proof}
We firstly discuss the case when $F_{1}(\cdot ; \mathbb{A})$ has a unique minimizer. 
Let $\{p_{j}\}_{j=1}^{\infty}$ be a real sequence with $p_{j} \in (1, 2]$ for all $j=1, 2, \cdots$ and $p_{j} \searrow 1$ as $j \to \infty$. 
We show that the sequence $\{M_{p_{j}}(\mathbb{A})\}_{j=1}^{\infty}$ is bounded in $(\mathbb{P}_{d}, \delta)$.  
Since $p_{j} \in (1, 2]$ for all $j=1, 2, \cdots$, 
\begin{align*}
F_{p_{j}}(M_{p_{j}}(\mathbb{A}) ; \mathbb{A}) 
\leq F_{p_{j}}(I ; \mathbb{A}) 
=\frac{1}{n}\sum_{k=1}^{n}\delta(I, A_{k})^{p_{j}} 
\leq \frac{1}{n}\sum_{k=1}^{n}(1+\delta(I, A_{k})^{2})
=:K
\end{align*}
holds for all $j=1, 2, \cdots$. 
Note that $K$ does not depend on $j$. 
We put $R:=1+\max_{1 \leq k \leq n} \delta(I, A_{k})$ and $S:=(1/n)\sum_{k=1}^{n} \delta(I, A_{k})$. 
If $j$ satisfies $\delta(M_{p_{j}}(\mathbb{A}), I)>R$, then, by the triangle inequality, 
\begin{align*}
\delta(M_{p_{j}}(\mathbb{A}), A_{k}) 
\leq \delta(M_{p_{j}}(\mathbb{A}), I)-\delta(I, A_{k}) 
&>R-\delta(I, A_{k}) \\
&=1+(\max_{1 \leq k \leq n} \delta(I, A_{k})-\delta(I, A_{k})) \\
&\geq 1
\end{align*}
holds for each $k=1, \cdots, n$. 
Hence, we obtain $\delta(M_{p_{j}}(\mathbb{A}), A_{k})^{p_{j}} \geq \delta(M_{p_{j}}(\mathbb{A}), A_{k})$ for all $j=1, 2, \cdots$, 
which yields 
\begin{align*}
K \geq F_{p_{j}}(M_{p_{j}}(\mathbb{A}) ; \mathbb{A}) 
=\frac{1}{n}\sum_{k=1}^{n}\delta(M_{p_{j}}(\mathbb{A}), A_{k})^{p_{j}} 
&\geq \frac{1}{n}\sum_{k=1}^{n}\delta(M_{p_{j}}(\mathbb{A}), A_{k}) \\
&\geq \frac{1}{n}\sum_{k=1}^{n}(\delta(M_{p_{j}}(\mathbb{A}), I)-\delta(I, A_{k})) \\
&=\delta(M_{p_{j}}(\mathbb{A}), I)-S.
\end{align*}
Thus, it follows that $\delta(M_{p_{j}}(\mathbb{A}), I) \leq K+S$, 
which implies that 
\[
\delta(M_{p_{j}}(\mathbb{A}), I) \leq C:=\max\{R, K+S\}
\]
holds for any $j=1, 2, \cdots$ whether $j$ satisfies $\delta(M_{p_{j}}(\mathbb{A}), I)>R$ or not. 
It is immediate that the closed ball $B(I, C):=\{X \in \mathbb{P}_{d} : \delta(X, I) \leq C\}$ is closed in a compact set 
$\{H \in \mathbb{H}_{d} : e^{-C}I \leq H \leq e^{C}I\}$, where $\mathbb{H}_{d}$ is all the $d \times d$ self-adjoint matrices, 
and hence, $B(I, C)$ is compact. 
The compactness of $B(I, C)$ guarantees that $\{M_{p_{j}}(\mathbb{A})\}_{j=1}^{\infty}$ has at least one convergent subsequence. 
We pick an arbitrary convergent subsequence $\{M_{p_{j_{l}}}(\mathbb{A})\}_{l=1}^{\infty}$ and 
suppose that $M_{p_{j_{l}}}(\mathbb{A}) \to M^{\prime}$ as $l \to \infty$. 
By definition, 
\begin{equation}\label{subpjl}
F_{{p_{j}}_{l}}(M_{{p_{j}}_{l}}(\mathbb{A}) ; \mathbb{A}) \leq F_{{p_{j}}_{l}}(X ; \mathbb{A})
\end{equation}
holds for any $X \in \mathbb{P}_{d}$ and $l=1, 2, \cdots$, 
$F_{{p_{j}}_{l}}(X ; \mathbb{A}) \to F_{1}(X ; \mathbb{A})$ as $l \to \infty$ by Lemma $\ref{convlocuni}$. 
By Lemma $\ref{convlocuni}$ again, 
$\{F_{p}( \cdot ; \mathbb{A})\}_{p \in (1, 2]}$ converges uniformly to $F_{1}( \cdot ; \mathbb{A})$ 
on a compact subset $\{M^{\prime}\} \cup \{M_{{p_{j}}_{l}}(\mathbb{A}) : l=1, 2, \cdots \}$. 
Hence, we obtain 
\begin{align*}
|F_{{p_{j}}_{l}}(M_{{p_{j}}_{l}}(\mathbb{A}) ; \mathbb{A})&-F_{1}(M^{\prime} ; \mathbb{A})| \\
&\leq 
|F_{{p_{j}}_{l}}(M_{{p_{j}}_{l}}(\mathbb{A}) ; \mathbb{A})-F_{1}(M_{{p_{j}}_{l}}(\mathbb{A}) ; \mathbb{A})|
+|F_{1}(M_{{p_{j}}_{l}}(\mathbb{A}) ; \mathbb{A})-F_{1}(M^{\prime} ; \mathbb{A})| 
\to 0
\end{align*}
as $l \to \infty$. 
Letting $l \to \infty$ in $\eqref{subpjl}$, we have $F_{1}(M^{\prime} ; \mathbb{A}) \leq F_{1}(X ; \mathbb{A})$ for any $X \in \mathbb{P}_{d}$, 
which means that $M^{\prime}=M_{1}(\mathbb{A})$ since we assume that $F_{1}( \cdot ; \mathbb{A})$ has a unique minimizer in the present case.  
Since $\{M_{p_{j}}(\mathbb{A}) : j=1, 2, \cdots \}$ is relatively compact, 
it follows that $M_{p_{j}}(\mathbb{A}) \to M_{1}(\mathbb{A})$ as $j \to \infty$. 

It remains to deal with the case that $n=2l$ for some positive integer $l$ and $A_{1}, \cdots, A_{n}$ lie on a common geodesic. 
In the present case, we may assume that the common geodesic is $\gamma(t)=A_{1} \#_{t/D} A_{n}, t \in [0, D]$, $A_{1} \neq A_{n}$,
there exist $t_{k} \in [0, D]$ such that $0=t_{1} \leq t_{2} \cdots \leq t_{n}=D$ and $\gamma(t_{k})=A_{k}$, 
where $D=\delta(A_{1}, A_{n})$. 
It can be shown that the minimizer $M_{p}(\mathbb{A})$ of $F_{p}(\cdot ; \mathbb{A})$ with $p>1$ belongs to $\gamma([0, D])$ 
in the same manner in Lemma $\ref{minongamma}$. 
Thus, minimizing $F_{p}( \cdot ; \mathbb{A})$ with $p>1$ on $\gamma([0, D])$ is equivalent to 
minimizing $g_{p}(t):=(1/(2l))\sum_{k=1}^{2l}|t-t_{k}|^{p}$ with $p>1$ on $[0, D]$. 
Moreover, since $F_{p}( \cdot ; \mathbb{A})$ with $p>1$ has a unique minimizer, 
there exists a unique minimizer $m_{p} \in [0, D]$ of $g_{p}(t)$ with $p>1$ such that $M_{p}(\mathbb{A})=\gamma(m_{p})$. 
In \cite{Lan81}, 
Landers and Rogge showed that the net $\{m_{p}\}_{p \in (1, 2]}$ converges to some $m \in [t_{l}, t_{l+1}]$ as $p \searrow 1$,  
which implies that the net $\{M_{p}(\mathbb{A})\}_{p \in (1, 2]}$ converges to $\gamma(m) \in \gamma([t_{l}, t_{l+1}])$. 
Since $m$ minimizes $g_{1}$, 
$\gamma(m)$ also minimizes $F_{1}( \cdot ; \mathbb{A})$. 
\end{proof}

Here is the Riemannian median that we propose. 
The validity of the definition is guaranteed by Lemma $\ref{mplimexists}$. 

\begin{dfn}\label{Riemed}
Let $\mathbb{A}=(A_{1}, \cdots, A_{n}) \in (\mathbb{P}_{d})^{n}$. 
Define the Riemannian median $M(\mathbb{A})$ of $\mathbb{A}$ by 
\[
M(\mathbb{A})=\lim_{p \searrow 1} M_{p}(\mathbb{A}), 
\]
where $M_{p}(\mathbb{A})$ denotes a unique minimizer of a map 
$F_{p}(X ; \mathbb{A})=(1/n)\sum_{k=1}^{n}\delta(X, A_{k})^{p}$ with $p>1$. 
\end{dfn}

We emphasize that $M(\mathbb{A})$ is explicitly constructed in terms of weighted geometric means, 
if $A_{1}, \cdots, A_{n}$ lie on a common geodesic. 


\section{SOME PROPERTIES OF THE RIEMANNIAN MEDIAN}

\subsection{The distance between the Riemannian median and the Karcher mean} 

It is well known that the mean $\mu$, the median $m$ and the standard deviation $\sigma$ of real-valued data satisfy the following inequality:
\begin{equation}\label{mms}
|\mu-m| \leq \sigma.
\end{equation}
We show an analogue of the above $\eqref{mms}$ in $(\mathbb{P}_{d}, \delta)$. 

\begin{thm}
Let $\mathbb{A}=(A_{1}, \cdots, A_{n}) \in (\mathbb{P}_{d})^{n}$. 
Then,  
\begin{equation}\label{main}
\delta(M, \Lambda) \leq \frac{1}{n}\sum_{k=1}^{n}\delta(A_{k}, \Lambda) \leq \sqrt{\frac{1}{n} \sum_{k=1}^{n} \delta(A_{k}, \Lambda)^{2}}, 
\end{equation}
holds, where $M=M(\mathbb{A})$ 
and $\Lambda$ is the Karcher mean of $\mathbb{A}$ with weight $[1/n, \cdots, 1/n]^{T}$. 
\end{thm}

\begin{proof}
Since the function $\mathbb{P}_{d} \ni X \mapsto \delta(M, X) \in [0, \infty)$ is continuous and convex, 
\begin{equation}\label{eq1}
\delta(M, \Lambda) \leq \frac{1}{n}\sum_{k=1}^{n}\delta(M, A_{k})
\end{equation}
follows by Proposition $\ref{Jensen}$. 
Since $M$ minimizes $F_{1}( \cdot ; \mathbb{A})$, 
we have 
\begin{equation}\label{eq2}
\frac{1}{n}\sum_{k=1}^{n}\delta(M, A_{k})
=F_{1}(M ; \mathbb{A}) 
\leq F_{1}(\Lambda ; \mathbb{A})
=\frac{1}{n}\sum_{k=1}^{n}\delta(\Lambda, A_{k}).
\end{equation}
By $\eqref{eq1}$, $\eqref{eq2}$ and the Cauchy-Schwarz inequality,  
we obtain 
\begin{align*}
\delta(M, \Lambda) 
\leq \frac{1}{n}\sum_{k=1}^{n}\delta(M, A_{k}) 
\leq \frac{1}{n}\sum_{k=1}^{n}\delta(\Lambda, A_{k}) 
\leq \sqrt{\frac{1}{n}\sum_{k=1}^{n} \delta(A_{k}, \Lambda)^{2}}. 
\end{align*}
Hence, we are done.
\end{proof}

\begin{rem}
Inequality $\eqref{main}$ holds even if $M$ is replaced 
by another minimizer of $F_{1}(\cdot ; \mathbb{A})$. 
\end{rem}


\subsection{Commutative case}

Let $\mathbb{A}=(A_{1}, \cdots, A_{n}) \in (\mathbb{P}_{d})^{n}$ and 
let $\mathbb{D}_{d}$ denote all the $d \times d$ diagonal matrix whose all entries are positive. 
Unless otherwise stated, $A_{1}, \cdots, A_{n}$ mutually commute in this subsection. 
We investigate the Riemannian median $M(\mathbb{A})$ when $A_{1}, \cdots, A_{n}$ mutually commute. 
Totally geodesic submanifolds play an important role in the case. 

\begin{dfn}[\cite{Bri99}]
A differential submanifold $L \subset \mathbb{P}_{d}$ is said to be totally geodesic 
if any geodesic line in $\mathbb{P}_{d}$ that intersects $L$ in two points is entirely contained in $L$. 
\end{dfn}

It is well known that $A_{1}, \cdots, A_{n}$ mutually commute if and only if there exists a unitary matrix $U$ that diagonalizes $A_{1}, \cdots, A_{n}$ simultaneously. 
We fix such a unitary $U$ and set 
\[
C_{U}=\{UDU^{\ast} : D \in \mathbb{D}_{d}\}. 
\]

\begin{lem}\label{totallygeo}
$C_{U}$ is totally geodesic in $\mathbb{P}_{d}$. 
\end{lem}

\begin{proof}
Let $\phi_{U}$ denote a linear isomorphism defined by 
\[
\phi_{U} :  \mathbb{P}_{d} \to \mathbb{P}_{d}, \ \phi_{U}(A)=UAU^{\ast}.
\]
It is easily seen that an image of a geodesic by $\phi_{U}$ is also a geodesic since the $t$-weighted geometric mean is invariant under congruence, 
and hence, it suffices to show that $\mathbb{D}_{d}$ is totally geodesic in $\mathbb{P}_{d}$. 

Suppose that a geodesic line $\gamma : \mathbb{R} \to \mathbb{P}_{d}$ intersects $\mathbb{D}_{d}$ in two points $A, B \in \mathbb{D}_{d}$. 
Recall that $\mathbb{P}_{d}$ is a Hadamard manifold. 
Hence, the geodesic joining $A$ and $B$ is unique.  
The velocity vector of this geodesic line at $A$ is $A^{1/2}\log(A^{-1/2}BA^{-1/2})A^{1/2}$, and thus,  
\begin{align*}
\gamma(t)
=A^{1/2}\exp(t A^{-1/2} (A^{1/2}\log(A^{-1/2}BA^{-1/2})A^{1/2})A^{-1/2})A^{1/2} 
&=A^{1/2}(A^{-1/2}BA^{-1/2})^{t}A^{1/2} \\
&=A^{1-t}B^{t}
\end{align*}
holds for all $t \in \mathbb{R}$. 
Observe that $A^{1-t}B^{t} \in \mathbb{D}_{d}$ if $A, B \in \mathbb{D}_{d}$. 
\end{proof}

It is easy to see that $C_{U}$ is convex and closed in $(\mathbb{P}_{d}, \delta)$. 

\begin{prop}
Let $\mathbb{A}=(A_{1}, \cdots, A_{n}) \in (\mathbb{P}_{d})^{n}$. 
If $A_{1}, \cdots, A_{n} \in \mathbb{P}_{d}$ mutually commute, then $M(\mathbb{A})$ and $A_{1}, \cdots, A_{n}$ also mutually commute. 
\end{prop}

We omit the detailed proof since we can show this proposition by almost the same way in Lemma 3.4; however, we give a sketch of the proof : 
We can show $\delta(M, A_{k})>\delta(\mathcal{P}(M), A_{k})$ holds for all $k=1, \cdots, n$, 
where $\mathcal{P}(M)$ is the orthogonal projection onto $C_{U}$. 
This yields $F_{1}(\mathcal{P}(M) ; \mathbb{A})<F_{1}(M ; \mathbb{A})$, which is a contradiction since $M$ is a minimizer of $F_{1}(\cdot ; \mathbb{A})$. 

By this proposition, 
we only have to investigate all the matrices that commute with $A_{1}, \cdots, A_{n}$ to find minimizers of $F_{1}(\cdot ; \mathbb{A})$.  
This fact means that the definition of Riemannian median in $\mathbb{P}_{d}$ is an extension of the one of the median of scalars. 


\subsection{Effect of outliers}

In this subsection, we explain that the Riemannian median is less sensitive to outliers than the Karcher mean. 
By \cite[Proposition 16]{Bha062}, 
\begin{equation}\label{nablad2}
\nabla_{X} \delta(X, A_{k})^{2}=-2X^{1/2} \log(X^{-1/2}A_{k}X^{-1/2}) X^{1/2}, \ k=1, \cdots, n
\end{equation}
holds for any $X \in \mathbb{P}_{d}$. 
Hence, putting 
\begin{align*}
u_{X}(A_{k})
&:=\frac{X^{1/2} \log(X^{-1/2}A_{k}X^{-1/2}) X^{1/2}}{\|X^{1/2} \log(X^{-1/2}A_{k}X^{-1/2}) X^{1/2}\|_{X}} \\
&=\frac{X^{1/2} \log(X^{-1/2}A_{k}X^{-1/2}) X^{1/2}}{\delta(X, A_{k})}, \ k=1, \cdots, n, 
\end{align*}
where $\| \cdot \|_{X}$ denotes the norm on $T_{X}\mathbb{P}_{d}$, 
we have 
\[
\nabla F_{2}(X ; \mathbb{A})=-\frac{2}{n}\sum_{k=1}^{n} \delta(X, A_{k})u_{X}(A_{k}), \ 
\nabla F_{1}(X ; \mathbb{A})=-\frac{1}{n}\sum_{k=1}^{n} u_{X}(A_{k}). 
\]
These mean that the contribution of each $A_{k}$ to $\nabla F_{2}$ is proportional to $\delta(X, A_{k})$, 
while its contribution to $\nabla F_{1}$ has norm one, independently of the Riemannian distance between $X$ and $A_{k}$. 
In particular, a data point far away from the Karcher mean $\Lambda$ of $\mathbb{A}$ contributes to the equation 
\[
\sum_{k=1}^{n} \delta(\Lambda, A_{k})u_{\Lambda}(A_{k})=0
\]
proportional to its (Riemannian) distance from $\Lambda$, 
on the other hand, each data point contributes to the equation
\[
\sum_{k=1}^{n} u_{M(\mathbb{A})}(A_{k})=0
\]
only through a unit tangent vector. 
Therefore, it can be said that the Riemannian median behaves more robustly against outliers than the Karcher mean. 


\subsection{Congruence invariance, joint homogeneity and self-duality}

In this subsection, 
let $C\mathbb{A}C^{\ast}$ denote a tuple $(CA_{1}C^{\ast}, \cdots, CA_{n}C^{\ast}) \in (\mathbb{P}_{d})^{n}$. 

Recall that $\mathbb{P}_{d}$ is identified with a homogeneous space ${\mathrm{GL}}_{d}(\mathbb{C})/U_{d}$ through the transitive action 
\[
\mathbb{P}_{d} \to \mathbb{P}_{d}, \ A \mapsto CAC^{\ast}, \ C \in {\mathrm{GL}}_{d}(\mathbb{C}), 
\] 
where ${\mathrm{GL}}_{d}(\mathbb{C})$ is all the $d \times d$ invertible matrices and $U_{d}$ is all the $d \times d$ unitary matrices. 
Thus, the Riemannian median should have the congruence invariance. 
In fact, it is easily seen that the Riemannian median has the property 
by the congruence invariance of $\delta$ and the uniqueness of $M_{p}(\mathbb{A})$ with $p>1$. 

\begin{thm}\label{conginv}
Let $\mathbb{A}=(A_{1}, \cdots, A_{n}) \in (\mathbb{P}_{d})^{n}$. 
Then, 
\[
M(C\mathbb{A}C^{\ast})=CM(\mathbb{A})C^{\ast}
\]
holds for any $C \in {\mathrm{GL}}_{d}(\mathbb{C})_{d}$.  
\end{thm}

By the congruence invariance, the joint homogeneity of the Riemannian median follows. 

\begin{cor}
Let $\mathbb{A}=(A_{1}, \cdots, A_{n}) \in (\mathbb{P}_{d})^{n}$. 
Then, 
\[
M(c\mathbb{A})=cM(\mathbb{A})
\]
holds for all $c>0$. 
Here, $c\mathbb{A}=(cA_{1}, \cdots, cA_{n})$.
\end{cor}

It is also immediate that the Riemannian median has the self-duality 
since $\delta(A^{-1}, B^{-1})=\delta(A, B)$ holds for all $A, B \in \mathbb{P}_{d}$ {\cite{Bha062}}.

\begin{thm}
Let $\mathbb{A}=(A_{1}, \cdots, A_{n}) \in (\mathbb{P}_{d})^{n}$. 
Then, 
\[
M(\mathbb{A}^{-1})^{-1}=M(\mathbb{A}). 
\] 
holds, where $\mathbb{A}^{-1}=(A_{1}^{-1}, \cdots, A_{n}^{-1})$. 
\end{thm}


\subsection{Monotonicity}

In this subsection, 
we will deal with the monotonicity of the Riemannian median. 
Let $\sigma(A)$ denote the spectrum of a matrix $A$. 
First, we show the monotonicity in a special case which corresponds to the case of considering the median on $\mathbb{R}$. 

\begin{thm}
Let $\mathbb{A}=(A_{1}, \cdots, A_{n}), \mathbb{B}=(B_{1}, \cdots, B_{n}) \in (\mathbb{P}_{d})^{n}$.  
If any $A_{i}$ and $B_{j}$ lie on a common geodesic $\gamma$ and 
$A_{k} \leq B_{k}$ for all $k=1, \cdots, n$, 
then 
\[
M(\mathbb{A}) \leq M(\mathbb{B}) 
\]
holds. 
\end{thm}

\begin{proof}
$If$ $A_{k}=B_{k}$ for every $k=1, \cdots, n$, then the conclusion is immediate
Hence, we may assume that there exists $k_{0}$ such that $A_{k_{0}} \neq B_{k_{0}}$. 
By this assumption, there exist $P, Q \in \{A_{1}, \cdots, A_{n}, B_{1}, \cdots, B_{n}\}$ with $P \neq Q$ and 
\[
\gamma(t)=P \#_{t/D} Q, \ t \in [0, D]
\]
such that $A_{1}, \cdots, A_{n}, B_{1}, \cdots, B_{n}$ lie on $\gamma([0, D])$, where $D:=\delta(P, Q)$. 
Here, there exist sequences $\{a_{k}\}_{k=1}^{n}$ and $\{b_{k}\}_{k=1}^{n}$ in $[0, D]$ such that 
$A_{k}=\gamma(a_{k}) \ \text{and} \ B_{k}=\gamma(b_{k})$ 
for all $k=1, \cdots, n$. 
Reversing $P$ and $Q$ if necessary, we may assume that $a_{k_{0}}<b_{k_{0}}$. 

We show that $s \leq t$ if and only if $\gamma(s) \leq \gamma(t)$. 
By assumption, we have 
\[
P^{1/2}(P^{-1/2}QP^{-1/2})^{a_{k_{0}}}P^{1/2}=A_{k_{0}} \leq B_{k_{0}}=P^{1/2}(P^{-1/2}QP^{-1/2})^{b_{k_{0}}}P^{1/2}, 
\]
which yields $(P^{-1/2}QP^{-1/2})^{a_{k_{0}}} \leq (P^{-1/2}QP^{-1/2})^{b_{k_{0}}}$. 
Since these two matrices commute, this is equivalent to $I \leq (P^{-1/2}QP^{-1/2})^{b_{k}-a_{k}}$. 
Thus, we obtain $P^{-1/2}QP^{-1/2} \geq I$ since $b_{k}-a_{k}>0$. 
Note that $P^{-1/2}QP^{-1/2} \neq I$ since $P \neq Q$. 
Hence, $s \leq t$ implies $(P^{-1/2}QP^{-1/2})^{s} \leq (P^{-1/2}QP^{-1/2})^{t}$, which is equivalent to $\gamma(s) \leq \gamma(t)$. 
The converse direction is trivial. 

Let $p>1$. 
It can be shown that the minimizer $M_{p}(\mathbb{A})$ of $F_{p}(\cdot ; \mathbb{A})$ belongs to $\gamma([0, D])$ 
in the same manner in Lemma $\ref{minongamma}$. 
Thus, there exist $m_{p, A}$ and $m_{p, B}$ in $[0, D]$ such that $M_{p}(\mathbb{A})=\gamma(m_{p, A})$ and $M_{p}(\mathbb{B})=\gamma(m_{p, B})$. 
On the other hand, since $\delta(\gamma(s), \gamma(t))=|s-t|$ for all $s, t \in [0, D]$, 
$F_{p}(X ; \mathbb{A})$ and $F_{p}(X ; \mathbb{B})$ on $\gamma([0, D])$ 
is equivalent to minimizing $f_{p}(x ; \{a_{k}\}_{k=1}^{n})=(1/n)\sum_{k=1}^{n}|x-a_{k}|$ and $f_{p}(x ; \{b_{k}\}_{k=1}^{n})=(1/n)\sum_{k=1}^{n}|x-b_{k}|$ on $[0, D]$. 
Thus, $m_{p, A}$ and $m_{p, B}$ are unique minimizers of $f_{p}(x ; a)$ and $f_{p}(x ; b)$, respectively. 

Observe that the function $\varphi_{p}$ with $p>1$ defined by $\varphi_{p}(x)={\rm sgn}(x)|x|^{p-1}$ is monotonically increasing. 
Since $t-a_{k} \geq t-b_{k}$ holds for $k=1, \cdots, n$ and $t \in [0, D]$, 
it follows that $\varphi_{p}(t-a_{k}) \geq \varphi_{p}(t-b_{k})$ holds for $k=1, \cdots, n$ and $t \in [0, D]$. 
Hence, 
\[
\frac{d}{dt}f_{p}(x ; \{a_{k}\}_{k=1}^{n})
=\frac{p}{n}\sum_{k=1}^{n}\varphi_{p}(t-a_{k})
\geq \frac{p}{n}\sum_{k=1}^{n}\varphi_{p}(t-b_{k})
=\frac{d}{dt}f_{p}(x ; \{b_{k}\}_{k=1}^{n})
\]
holds for any $x \in [0, D]$. 
Especially, 
\[
\frac{d}{dt}f_{p}(m_{p, A} ; \{b_{k}\}_{k=1}^{n}) \leq \frac{d}{dt}f_{p}(m_{p, A} ; \{a_{k}\}_{k=1}^{n})
=0 
=\frac{d}{dt}f_{p}(m_{p, B} ; \{b_{k}\}_{k=1}^{n})
\]
holds. 
Since $\frac{d}{dt}f_{p}(x ; \{b_{k}\}_{k=1}^{n})$ is monotonically increasing and $m_{p, B}$ is its unique zero, 
we obtain $m_{p, A} \leq m_{p, B}$, and hence, 
\[
M_{p}(\mathbb{A})=\gamma(m_{p, A}) \leq \gamma(m_{p, B})=M_{p}(\mathbb{B})
\] 
holds for all $p>1$. 
Therefore, by Definition $\ref{Riemed}$, it follows that $M(\mathbb{A}) \leq M(\mathbb{B})$. 
\end{proof}

We give an example which shows the monotonicity of the Riemannian median does not hold in general. 
Recall that the Fermat point of a triangle $\triangle \mathrm{ABC}$ in $\mathbb{R}^{2}$ 
is a point $\mathrm{M}$ which minimizes $\mathrm{AM}+\mathrm{BM}+\mathrm{CM}$. 
See \cite{Spa96} for more details of the Fermat point.  
Here is the counter example. 

\begin{exm}
Let 
\begin{align*}
P_{A}:=\begin{bmatrix} 1 & 0 \\ 0 & 1 \end{bmatrix}, \ 
P_{B}:=\begin{bmatrix} e^{2} & 0 \\ 0 & e^{4} \end{bmatrix}, \ 
P_{C}:=\begin{bmatrix} e^{10} & 0 \\ 0 & 1 \end{bmatrix}, \ 
P_{D}:=\begin{bmatrix} e^{10} & 0 \\ 0 & e^{10} \end{bmatrix}. 
\end{align*}
Since these matrices commute, we have 
\[
M(P_{A}, P_{B}, P_{C})=\begin{bmatrix} e^{\frac{14+10\sqrt{3}}{13}} & 0 \\ 0 & e^{\frac{54+20\sqrt{3}}{39}} \end{bmatrix}, \ 
M(P_{A}, P_{B}, P_{D})=\begin{bmatrix} e^{2} & 0 \\ 0 & e^{4} \end{bmatrix}. 
\]
We obtain these medians by computing the Fermat points of triangles in $\mathbb{R}^{2}$ 
with their vertices $(0, 0), (2, 4), (10, 0)$ and $(0, 0), (2, 4), (10, 10)$, respectively. 
Here, observe that $(14+10\sqrt{3})/13>2$, and hence $M(P_{A}, P_{B}, P_{C}) \leq M(P_{A}, P_{B}, P_{D})$ \underline{does not hold} 
though $P_{C} \leq P_{D}$ hold. 
This means that the monotonicity of the Riemannian median does not hold in general. 
\end{exm}


\subsection{Weighted Riemannian median}

Let $w \in \mathbb{R}^{n}$ with $\sum_{k=1}^{n}w_{k}=1$ and $w_{k}>0$ for all $k=1, \cdots, n$. 
It is shown that minimizers of the function 
\[
F_{1}(X ; \mathbb{A}, w)=\sum_{k=1}^{n}w_{k}\delta(X, A_{k}). 
\]
exist by Lemma $\ref{pfofexmnmzr}$. 
Using Landers and Rogges's work \cite{Lan81} and 
the existence of and uniqueness of minimizer $M_{p}(w ; \mathbb{A})$ of the function $F_{p}(X ; \mathbb{A} ; w):=\sum_{k=1}^{n}w_{k}\delta(X, A_{k})^{p}$,   
we can show that $\lim_{p \searrow 1}M_{p}(w ; \mathbb{A})$ exists. 
Therefore, it is also possible to define the Riemannian median with weight. 

Here, unlike the case that weight is $[1/n, \cdots, 1/n]$, 
note that the uniqueness of a minimizer of $F_{1}(\cdot ; \mathbb{A}, w)$ cannot be shown 
even if $n$ is odd when $A_{1}, \cdots, A_{n}$ lie on a common geodesic and 
there exists an $m$ such that $\sum_{k=1}^{m}w_{k}=\sum_{k=m+1}^{n}w_{k}=1/2$.  
Also, note that the uniqueness of minimizer of $F_{1}(\cdot ; \mathbb{A}, w)$ can be shown 
even if $n$ is even and $A_{1}, \cdots, A_{n}$ lie on a common geodesic when 
there does not exist $m$ such that $\sum_{k=1}^{m}w_{k}=\sum_{k=m+1}^{n}w_{k}=1/2$. 
Moreover, by some modification, 
we can show the congruence invariance, the self-duality, the monotonicity in the case that all data points lie on a common geodesic 
and inequality
\[
\delta(M(w ; \mathbb{A}), \Lambda(w ; \mathbb{A})) \leq \sqrt{\sum_{k=1}^{n}w_{k}\delta(\Lambda(w ; \mathbb{A}), A_{k})}, 
\]
where $M(w ; \mathbb{A})=\lim_{p \searrow 1}M_{p}(w ; \mathbb{A})$ and $\Lambda(w ; \mathbb{A})$ is the $w$-weighted Karcher mean. 


{
\renewcommand{\addcontentsline}[3]{}

\section*{Acknowledgements}

The author would like to express my gratitude to Professor Yoshimichi Ueda for editorial supports, comments, 
his encouragement and giving the author information of references.
The author gratefully acknowledges Professor Fumio Hiai who kindly read a previous version of these notes 
and gave many fruitful comments, which enabled the author to improve this work. 
}


\end{document}